\documentclass[10pt,twoside]{amsart}
\usepackage{pstricks}
\usepackage[bookmarks=true]{hyperref}
\usepackage{subcaption}

\newtheorem{thm}{Theorem}[section]
\newtheorem{prop}[thm]{Proposition}
\newtheorem{cor}[thm]{Corollary}
\newtheorem{lem}[thm]{Lemma}
\theoremstyle{definition}
\newtheorem{defn}[thm]{Definition}

\newtheorem*{Acknowledgements}{Acknowledgements}

\newtheorem{rem}[thm]{Remark}

\newcommand{\hgt}{{\rm ht}\,}
\newcommand{\bight}{{\rm bight}\,}
\newcommand{\ara}{{\rm ara}\,}
\newcommand{\pd}{{\rm pd}\,}

\title[]{The arithmetical rank of the edge ideals \\ of cactus graphs}

\author{Margherita Barile}
\address[M. Barile]{Dipartimento di Matematica, Universit\`a degli Studi di Bari ``Aldo Moro'', Via Orabona 4, 70125 Bari, Italy}
\email{margherita.barile@uniba.it}

\author{Antonio Macchia}
\address[A. Macchia]{Dipartimento di Matematica, Universit\`a degli Studi di Bari ``Aldo Moro'', Via Orabona 4, 70125 Bari, Italy}
\email{macchia.antonello@gmail.com}

\begin{document}

\begin{abstract}
We prove that, for the edge ideal of a cactus graph, the arithmetical rank is bounded above by the sum of the number of cycles and the maximum height of its associated primes. The bound is sharp, but in many cases it can be improved. Moreover, we show that the edge ideal of a Cohen-Macaulay graph that contains exactly one cycle or is chordal or has no cycles of length 4 and 5 is a set-theoretic complete intersection.
\end{abstract}

\maketitle



\section{Introduction}

The main object of our work is the so-called {\it edge ideal} $I(G)$ of a (non-empty) finite simple graph $G$. It is defined by considering the polynomial ring $R$ over a given field whose indeterminates are the vertices of $G$, and then taking the (squarefree monomial) ideal generated by the products $xy$ such that $\{x,y\}$ is an edge of $G$.  This algebraic-combinatorial construction dates back to 1990: it appears for the first time in a paper by Villarreal \cite{V90}. The corresponding theory was developed in 1994 by Simis, Vasconcelos and Villarreal \cite{SVV}, and, in the following years, by many other authors: an extensive overview of the results achieved can be found in \cite{MV11}. The present paper deals with two algebraic invariants of $I(G)$: the \textit{big height}, denoted by $\bight I(G)$, which is the maximum height of the minimal primes of $I(G)$, and the \textit{arithmetical rank}, denoted by $\ara I(G)$, which is the minimum number of elements of $R$ that generate an ideal whose radical is $I(G)$ (we will say that these elements generate $I(G)$ {\it up to radical}). As a consequence of Krull's Principal Ideal Theorem, the big height always provides a lower bound for the arithmetical rank. In our case, more precisely we have:
\begin{equation}\label{0}
\hgt I(G)\leq\bight I(G)\leq \pd R/ I(G)\leq \ara I(G),
\end{equation}
where $\hgt$ and $\pd$ denote the height and the projective dimension, respectively. These inequalities are strict, in general.  In \cite{BM2} we showed that, if $G$ is a graph whose cycles are pairwise vertex-disjoint, then the difference between $\ara I(G)$ and $\bight I(G)$ is not greater than the number of cycles of $G$. This upper bound turns out to be sharp. In Theorem \ref{thm_main} we will extend this result to all graphs whose cycles are pairwise edge-disjoint, i.e., have at most one vertex in common. These are known in the literature as {\it cactus graphs}. Our approach is inductive on the number of edges, exploits the main result in \cite{M13b}, and its proof is entirely independent of the one in \cite{BM2}. We note that our Theorem \ref{thm_main} gives a further generalization of the theorem by Kimura and Terai \cite{KT13}, according to which the arithmetical rank and the big height are equal for acyclic graphs (the so-called {\it forests}), and which, in turn, is a generalization of the main result in \cite{B}.

If in (\ref{0}) equality holds everywhere, the ideal $I(G)$ is called a {\it set-theoretic complete intersection}; in this case, the ideal is Cohen-Macaulay. On the other hand, there are many classes of edge ideals for which the converse is also true. Some of them are characterized by certain algebraic properties of $I(G)$, such as, e.g., having height two \cite{K} or having height equal to half the number of vertices \cite{BM}, whereas other results refer to the combinatorial features of $G$ (see the bipartite graphs studied in \cite{HH05} and in \cite{EV97}, which also belong to the aforementioned class). Here we will show that the edge ideal of a Cohen-Macaulay graph that is chordal or has no cycles of length 4 and 5 (Corollary \ref{cor_chordal}), or contains exactly one cycle (Theorem \ref{thm_1cycle}), is a set-theoretic complete intersection; our proofs are based on the combinatorial characterizations provided in \cite{TV90}, \cite{HHZ05} and \cite{HMT15}.

Another special case that is worth considering is the one where $\pd R/I(G)=\ara I(G)$. This equality has been established for certain unmixed bipartite graphs \cite{Ku09}, for the graphs formed  by one cycle or by two cycles having one vertex in common \cite{BKMY12}, for those formed by some cycles and lines having a common vertex \cite{KM12}, or those whose edge ideals are subject to certain algebraic constraints (see, e.g., \cite{EOT10} and \cite{KRT}).

A stronger condition is the equality between the arithmetical rank and the big height, which has been proven for acyclic graphs \cite{KT13}, for graphs formed by a single cycle and some terminal edges attached to some of its vertices (\textit{multiwhisker graphs on a cycle}, see \cite{M13b} or \cite{M13c}), and  for graphs in which every vertex  belongs to a terminal edge (\textit{fully whiskered graphs}, (see \cite{M13a} or \cite{M13c}). Further examples in this direction arise from our Corollary \ref{cor_3} (graphs in which every cycle has length divisible by 3 and has all but two consecutive vertices of degree 2) and Proposition \ref{prop_bubble} (graphs obtained by attaching, to each vertex of a given graph, a single edge or a cycle whose length is not congruent to 1 modulo 3). Both results, which could be considered as interesting for themselves,  provide examples of set-theoretic complete intersections.

All the results proven in this paper hold on any field.

\section{Preliminaries}

We first introduce some graph-theoretical terminology and notation.

All graphs considered in this paper are simple, i.e., without multiple edges or loops. For a graph $G$, $V(G)$ and $E(G)$ will denote the set of vertices and the set of edges of $G$, respectively. Given two vertices $x$ and $y$ of $G$, we will say that $x$ is a \textit{neighbour} of $y$ if the vertices $x,y$ form an edge. By abuse of notation, this edge will be denoted by $xy$, with the same symbol used for the corresponding monomial of $I(G)$. The vertex $x$ will be called \textit{terminal} if it has exactly one neighbour $y$; in this case the edge  $xy$ will be called \textit{terminal}. For the remaining basic terminology about graphs we refer to \cite{Ha}.

A graph will always be identified with the set of its edges. In particular, if $G$ and $H$ are graphs, the {\it union} of $G$ and $H$ (denoted $G\cup H$) will be the graph with vertex set $V(G)\cup V(H)$ and edge set $E(G)\cup E(H)$. The graphs $G$ and $H$ will be called disjoint if so are their sets of edges. We will say that $H$ is a {\it subgraph} of $G$ if $V(H)\subset V(G)$ and $E(H)\subset E(G)$. In this case, $E(G\setminus H)=E(G)\setminus E(H)$.

\begin{defn} Let $G$ be a graph. A subset $C$ of its vertex set is called a \textit{vertex cover} if all edges of $G$ have a vertex in $C$. A vertex cover of $G$ is called \textit{minimal} if it does not properly contain any vertex cover of $G$. A minimal vertex cover of $G$ is called \textit{maximum} if it has maximum cardinality among the minimal vertex covers of $G$. The empty set is the maximum minimal vertex cover of any graph with no edges.
\end{defn}

\begin{rem} It is well known that the minimal vertex covers of $G$ are the sets of generators of the minimal primes of $I(G)$. Hence $\bight I(G)$ is the cardinality of the maximum minimal vertex covers of $G$.
\end{rem}

\begin{defn}
Let $G$ be a graph, and $H$ a subgraph of $G$.
If $C$ is a minimal vertex cover of $G$, we will say that the (possibly empty) set  $C\cap V(H)$ is the vertex cover \textit{induced} by $C$ on $H$.
\end{defn}

\begin{defn} Let $G$ be a graph and $x$ be one of its vertices. Let $C$ be a minimal vertex cover of $G$ such that $x\notin C$. Then a neighbour $y$ of $x$ will be called \textit{redundant} (\textit{in $C$}) if all neighbours of $y$ other than $x$ belong to $C$.

\end{defn}

\begin{rem}\label{remark2}
Let $C$ be a minimal vertex cover of $G$ such that $x\notin C$. Then there is a redundant neighbour $y$ of $x$ in $C$ if and only if $C \setminus \{y\}$ is a minimal vertex cover of $G \setminus \{xy\}$.
\end{rem}

\begin{lem}\label{lem_max}
Let $G$ be a graph and let $G_1$ be a subgraph of $G$ such that the vertex sets of $G_1$ and $G\setminus G_1$  have exactly one element $x$ in common. Let $C$ be a maximum minimal vertex cover of $G$. Let $C_1$ be a
 maximum minimal vertex cover of $G_1$, and suppose that $x\in C_1$ for all choices of $C_1$. Let $D_1$ be the vertex cover induced by $C$ on $G_1$. Then $\vert D_1\vert\leq\vert C_1\vert$. If $x\in C$, equality holds.
\end{lem}
\begin{proof}   If $D_1$ is a minimal vertex cover of $G_1$, the inequality is true. If it is not minimal, then minimality must be violated at $x$, and in particular $x\in D_1$. Hence all neighbours of $x$ in $G_1$ belong to $D_1$. But then $\overline{D_1}=D_1\setminus\{x\}$ is a minimal vertex cover of $G_1$ (not containing $x$), so that, by assumption, $\vert\overline{D_1}\vert<\vert C_1\vert$, i.e., $\vert D_1\vert\leq \vert C_1\vert$, as claimed. Now suppose that $x\in C$. If $D_2$ is the vertex cover induced by $C$ on $G\setminus G_1$, then $C=D_1\cup D_2$, and $D_2$ does not contain any vertex of $G_1$ other than $x$. Note that $C_1\cup D_2$ is also a minimal vertex cover of $G$. Since $C$ is maximum, we must have $\vert D_1\vert\geq\vert C_1\vert$, which provides the required equality.
\end{proof}

\begin{lem}\label{lem_covers}
Let $G_1$ and $G_2$ be graphs whose vertex sets have exactly one element $x$ in common. Let $C_1$ and $C_2$ be maximum minimal vertex covers of $G_1$ and $G_2$, respectively. Suppose that
\begin{list}{}{}
\item[(i)] for all choices of $C_1$ and $C_2$, $x$ belongs to $C_1$ and $C_2$, or
\item[(ii)] for some choice of $C_1$ and $C_2$, $x$ does not belong to $C_1$, nor to $C_2$, or
\item[(iii)] for all choices of $C_1$, $x\in C_1$, and, for some choice of $C_2$,  $x\notin C_2$ and in $C_2$ there are no redundant neighbours of $x$.
\end{list}
Then $C_1\cup C_2$ is a maximum minimal vertex cover of $G_1\cup G_2$. Moreover,  in case (i), $x$ belongs to all maximum minimal vertex covers of $G_1\cup G_2$. In cases (ii) and (iii),  $C_1$ and $C_2$ are disjoint.
\end{lem}

\begin{proof} Case (i) is Lemma 3.4 in \cite{BM2}. Case (ii) is clear. Let us consider case (iii). Note that $C=C_1\cup C_2$ is a vertex cover of $G=G_1\cup G_2$. It is minimal because minimality could only be violated at some neighbour $y$ of $x$ in $C_2$, but this is not the case, because, by assumption, not all neighbours of $y$ in $G_2$ other than $x$ belong to $C_2$. By Lemma \ref{lem_max} the vertex cover induced by any maximum minimal vertex cover $C'$ of $G$ on $G_1$ has at most $\vert C_1\vert$ elements. The remaining elements of $C'$ form a minimal vertex cover of a subgraph of $G_2$, hence they cannot be more than $\vert C_2\vert$. This implies that $C$ is maximum.
\end{proof}

\section{On cactus graphs}
In this section we present our main result. Its proof will be performed by induction on the number of edges, after splitting the given cactus graph in (full) subgraphs having exactly one vertex in common.

\begin{defn}\label{def_cactus}
A {\it cactus graph} is a graph in which any two cycles have at most one vertex in common.
\end{defn}

\begin{defn}\label{def_branch} Let $G$ be a connected graph and $x$ be one of its vertices. We call a {\it branch} of $G$ at $x$ any maximal connected subgraph $H$ of $G$ such that $x$ is one of its vertices and either
\begin{list}{}{}
\item[(i)] $x$ is a terminal vertex of $H$, or
\item[(ii)] $x$ lies on a cycle of $H$ and has degree 2 in $H$.
\end{list}

In case (i) and (ii) $H$ will be called a 1-branch or a 2-branch, respectively.
\end{defn}

\begin{lem}\label{lem_branch} Let $G$ be a cactus graph, and let $x$ be one of its vertices. Let $H_1$ and $H_2$ be distinct branches of $G$ at $x$. Then $V(H_1)\cap V(H_2)=\{x\}$. Moreover, if $G$ is connected, it is the union of the branches of $G$ at $x$.
\end{lem}

\begin{proof} Suppose for a contradiction that some vertex $w$ of $G$, distinct from $x$, belongs to $V(H_1)$ and $V(H_2)$. Let $H=H_1\cup H_2$. Then $H$ is a connected subgraph of $G$ having $x$ as one of its vertices. Moreover, in each of $H_1$ and $H_2$ there is a path connecting $x$ to $w$. Let $x, y_1, \dots, y_m, w$ and $x, z_1, \dots, z_n, w$ be the vertices of these paths, which we suppose simple, i.e.,  without self-intersections. Moreover, set
\[
\begin{array}{c}
s=\min\{k\vert y_k=z_h\text{ for some }h\},\\[1mm]
t=\min\{h\vert z_h=y_k\text{ for some }k\}.
\end{array}
\]

First assume that $s=t=1$. If $H_1$ and $H_2$ are both 1-branches at $x$, then $y_1=z_1$ is the only neighbour of $x$ in $H$. Thus the vertex $x$ has degree 1 in $H$, against the maximality of $H_1$ and $H_2$.

Otherwise, if say $H_1$ is a 2-branch at $x$, and $x$ lies on the cycle $\Gamma_1$ of $H_1$, then there is a path from $x$ to $w$ whose second vertex is the neighbour $y'_1\neq y_1$ of $x$ on $\Gamma_1$: it is formed by  $xy'_1$ and a path connecting $y'_1$ to $w$ in $H_1\setminus\{xy'_1\}$. If $y'_1$ is a vertex of $H_2$, then it is a neighbour of $x$ other than $z_1$, so that $H_2$ is a 2-branch. In this case $H$ contains the cycle $\Gamma_1$ and in $H$ the vertex $x$ has degree 2, which violates the maximality of $H_1$ and $H_2$. Thus $y'_1$ is not a vertex of $H_2$, so that (up to changing the name of $y'_1$ to $y_1$) we have $s>1$. We thus may assume that $s>1$ or $t>1$, in which case the vertices $x, y_1,\dots, y_s =z_t, z_{t-1},\dots, z_1$ form a cycle $\Gamma$ in $H$. In order to produce the final contradiction it suffices to prove that in $H$ the vertex $x$ has degree 2, because this, once again, implies a violation of maximality. So suppose that in $H$ the vertex $x$ has degree greater than 2, i.e., it has a neighbour $u$ other than $y_1$ and $z_1$. Then, up to a change of indices, we may assume that $u$ is a vertex of $H_1$. Thus $H_1$ is a 2-branch and $u$ lies on the cycle $\Gamma_1$. Now, since $\Gamma_1$ and $\Gamma$ have the vertices $u$ and $x$ in common, they must coincide. But then $u$ is a neighbour of $x$ in $\Gamma$, i.e., $u\in \{y_1, z_1\}$, against our assumption. This completes the proof of the first part of the claim.

Now suppose that $G$ is connected and let $uv$ be an edge of $G$. We show that it belongs to some branch of $G$ at $x$. This is clear if $x=u$. So assume that $x\ne u$. Then there is a simple  path $x, z_1,\dots, z_m=u$, in which $x\ne z_1$. If $xz_1$ is contained in some cycle, then $uv$ belongs to a 2-branch of $G$ at $x$, which is the maximal connected subgraph of $G$ that contains this cycle and in which $x$ has degree 2. Otherwise it belongs to a $1$-branch of $G$ at $x$, namely the maximal connected subgraph of $G$ that contains $xz_1$ and in which $x$ has degree 1.
\end{proof}

\begin{thm}\label{thm_main}
Let $G$ be a cactus graph and let $n$ be the number of its cycles. Then $\ara I(G)\leq\bight I(G)+n$.
\end{thm}

\begin{proof}
The claim is trivial if $G$ consists of one single edge. So assume that $\vert E(G)\vert \geq 2$ and the claim is true for cactus graphs having a smaller number of edges. Since the arithmetical rank, the big height and the number of cycles are additive on connected components, we may assume that $G$ is connected. As remarked at the beginning of the proof of Theorem 4.8 in \cite{BM2}, we may assume that all vertices lying on a cycle of $G$ have degree greater than 2: in view of Lemma 3.1 in \cite{BM2}, each time a cycle of a graph is opened at a vertex of degree 2, the arithmetical rank of the edge ideal does not decrease, whereas its big height increases at most by one, so that the required upper bound remains unchanged.  Moreover, if all vertices of $G$ belong to some terminal edge, then the claim is true, since in \cite{M13b} (see Corollary 1 and Remark 3), it was proven that in this case $\ara I(G)=\bight I(G)$ ($G$ is a fully whiskered graph). So suppose that $G$ has a vertex $x$ that does not belong to any terminal edge.
Let $G_2$ be one of the branches of $G$ at $x$, and let $G_1=G\setminus G_2$.  Then, by assumption, $G_1$ is not empty, $G_2$ is not a single edge and, by Lemma \ref{lem_branch}, $G_1$ and $G_2$ have only the vertex $x$ in common. If $G_2$ is a 1-branch, let $y_1$ be the only neighbour of $x$ in $G_2$, if $G_2$ is a 2-branch, let $y_1$ and $y_2$ be the neighbours of $x$ in $G_2$. Let $C_1$ and $C_2$ be maximum minimal vertex covers of $G_1$ and $G_2$, respectively, and, for $i=1,2$, let $b_i=\vert C_i\vert$, and let $n_i$ be the number of cycles in $G_i$, so that $n=n_1+n_2$. Let $b$ be the cardinality of any maximum minimal vertex cover of $G$.
We distinguish several cases.

{\bf Case 1} First suppose that $x\in C_1$ for all choices of $C_1$.

{\bf Case 1.1} If $x\in C_2$ for all choices of $C_2$, then, by Lemma \ref{lem_covers} (i), $C=C_1\cup C_2$ is a maximum minimal vertex cover of $G$ and thus $b=b_1+b_2-1$. If $G_2$ is a 1-branch, set $G'_1=G_1\cup\{xy_1\}$, otherwise set $G'_1=G_1\cup\{xy_1, xy_2\}$. Let $C'_1$ be a maximum minimal vertex cover of $G'_1$, and set $b'_1=\vert C'_1\vert$. Moreover, if $G_2$ is a 1-branch, let $\overline{G}_2=G_2\setminus\{xy_1\}$ (which has $n_2$ cycles), otherwise set $\overline{G}_2=G_2\setminus\{xy_1, xy_2\}$ (which has $n_2-1$ cycles), so that, in both cases, $G$ is the disjoint union of $G'_1$ and $\overline{G}_2$. Let $\overline{C}_2$ be a maximum minimal vertex cover of $\overline{G}_2$, and set $\overline{b}_2=\vert\overline{C}_2\vert$. Note that $\overline{G}_2$ is not empty.

First suppose that $G_2$ is a 1-branch. Then, in view of  Lemma 3.2 in \cite{BM2}, $b'_1=b_1$.
If $y_1\in \overline{C}_2$, then $\overline{C}_2$ is a minimal vertex cover of $G_2$ (not maximum because $x\notin\overline{C}_2$), otherwise so is $\overline{C}_2\cup\{x\}$. Hence, in any case, $\overline{b}_2\leq b_2-1.$ Induction applies to $G'_1$ and $\overline{G}_2$. Thus
\begin{eqnarray}\label{1}
\ara I(G)&\leq&\ara I(G'_1)+\ara I(\overline{G}_2)\\
\nonumber&\leq&b'_1+n_1+\overline{b}_2+n_2
\end{eqnarray}

\begin{eqnarray}
\nonumber&\leq& b_1+n_1+b_2-1+n_2\\
\nonumber&=& b+n,
\end{eqnarray}
as was to be proved.

Now suppose that $G_2$ is a 2-branch. Let $D_1$ be the vertex cover induced by $C'_1$ on $G_1$. If $x\in C'_1$, then none of $y_1$ and $y_2$ belongs to $C'_1$, so that $C'_1=D_1$. Hence, by Lemma \ref{lem_max}, we have $\vert C'_1\vert=\vert C_1\vert$, i.e., $b'_1=b_1$.  In this case $C_1$, that is always a minimal vertex cover of $G'_1$, is also maximum. So assume that it is not maximum. Then  $x\notin C'_1$, and  we must have $C'_1=D_1\cup\{y_1, y_2\}$, with  $\vert D_1\vert\geq b_1-1$. Note that $D_1$ is a minimal vertex cover of $G_1$, but, since $x\notin D_1$, it is not maximum. Thus $\vert D_1\vert= b_1-1$, so that in this case $b'_1= b_1+1$. In any case $b'_1\leq b_1+1$.

If $y_1, y_2\in \overline{C}_2$, then $\overline{C}_2$ is a minimal vertex cover of $G_2$ (not maximum because $x\notin\overline{C}_2$), otherwise so is $\overline{C}_2\cup\{x\}$. Hence, in any case, $\overline{b}_2\leq b_2-1.$ Induction applies to $G'_1$ and $\overline{G}_2$. Thus
\begin{eqnarray}\label{2}
\ara I(G)&\leq&\ara I(G'_1)+\ara I(\overline{G}_2)\\
\nonumber&\leq&b'_1+n_1+\overline{b}_2+n_2-1\\
\nonumber&\leq& b_1+1+n_1+b_2-1+n_2-1\\
\nonumber&=& b+n,
\end{eqnarray}
as was to be proved.

{\bf Case 1.2} Now suppose that, for some choice of $C_2$, $x\notin C_2$, whence $y_1\in C_2$ if $G_2$ is a 1-branch, and $y_1, y_2\in C_2$ if $G_2$ is a 2-branch.

{\bf a)} First assume that for one of these choices, there are no redundant neighbours of $x$ in $C_2$. In this case, by Lemma \ref{lem_covers} (iii), $C=C_1\cup C_2$ is a maximum minimal vertex cover of $G$. Since $\vert C\vert=b_1+b_2$, the claim follows by induction applied to $G_1$ and $G_2$.

{\bf b)} Now assume that for all choices of $C_2$ such that $x\notin C_2$,  $x$ has some redundant neighbour in $C_2$.
\newline
We first show that, in this case,
\begin{equation}\label{bb1b2}
b\leq b_1+b_2-1.
\end{equation}

Let $C$ be a maximum minimal vertex cover of $G$. If $x\in C$, according to Lemma \ref{lem_max},  the cover $C_1$ induced on $G_1$ has exactly $b_1$ elements. Let $D_2$ be the cover induced by $C$ on $\overline{G}_2$. Then $C=C_1\cup D_2$, where the union is disjoint, and, if $y_i\in D_2$, then $y_i$ is a non-redundant neighbour of $x$ in $D_2$. If all available $y_i$ belong to  $D_2$, then $D_2$ is a minimal vertex cover of $G_2$, not maximum by assumption. Hence $\vert D_2\vert\leq b_2-1$. Otherwise $D_2\cup\{x\}$ is a minimal vertex cover of $G_2$, whence again $\vert D_2\vert\leq b_2-1$. Thus $\vert C\vert\leq b_1+b_2-1$. Now suppose that $x\notin C$. Then the cover $D_1$ induced on $G_1$ has at most $b_1-1$ elements, whereas the cover $D_2$ induced on $G_2$ (equivalently, on $\overline{G}_2$), has at most $b_2$ elements. Thus, $C=D_1\cup D_2$, where the union is disjoint, and once again, $b=\vert C\vert =\vert D_1\vert +\vert D_2\vert\leq b_1+b_2-1$.  This completes the proof of (\ref{bb1b2}).

Next suppose that $G_2$ is a 1-branch. If $y_1\notin \overline{C}_2$, then $\overline{C}_2\cup\{x\}$ is a minimal vertex cover of $G_2$, otherwise so is $\overline{C}_2$, but it cannot be maximum, because in it $y_1$ is a non-redundant neighbour of $x$. Hence $\overline{b}_2\leq b_2-1$.

Now suppose that $G_2$ is a 2-branch. If $y_1\notin \overline{C}_2$ or $y_2\notin\overline{C}_2$, then $\overline{C}_2\cup\{x\}$ is a minimal vertex cover of $G_2$, otherwise so is $\overline{C}_2$, but it cannot be maximum, because in it $y_1$ and $y_2$ are both non-redundant neighbours of $x$.  Hence we always have $\overline{b}_2\leq b_2-1$.

If, for some choice of $C_2$, $y_1$ is the only redundant neighbour of $x$ in $C_2$ (which is certainly true if $G_2$ is a 1-branch), then a maximum minimal vertex cover of $G$ is
$C_1\cup C_2\setminus\{y_1\}$, of cardinality $b=b_1+b_2-1$. Since the bound on $b'_1$ is unchanged,  the claim thus follows as in (\ref{1}) and (\ref{2}), by induction on $G'_1$ and $\overline{G}_2$.

Finally assume that $G_2$ is a 2-branch and, for all choices of $C_2$ such that $x\notin C_2$, both $y_1$ and $y_2$ are redundant neighbours of $x$. Now, recall that $b'_1\leq b_1+1$. Moreover,
\begin{itemize}
\item[-] $\overline{C}_2\cup\{x\}$ is a minimal vertex cover of $G_2$ if one of $y_1$ and $y_2$ does not belong to $\overline{C}_2$;
\item[-] $\overline{C}_2$ is a minimal vertex cover of $G_2$, but not maximum, if $y_1$ and $y_2$ both belong to $\overline{C}_2$.
\end{itemize}

In any case, $\vert\overline{C}_2\vert\leq b_2-1$. Hence, if $b=b_1+b_2-1$, the claim follows as above by induction on $G'_1$ and $\overline{G}_2$.

So assume that $b<b_1+b_2-1$. Note that a minimal vertex cover of $G$ can be obtained by taking $C=C_1\cup C_2\setminus\{y_1, y_2\}$. Thus $b=b_1+b_2-2$. We also have that no minimal vertex cover of $\overline{G}_2$ has cardinality $b_2-1$: if there were such a cover $\overline{C}_2$, then   $C_1\cup\overline{C}_2$ would be a minimal vertex cover of $G$ of cardinality $b_1+b_2-1$, against our present assumption. Hence $\overline{b}_2\leq b_2-2$. Thus induction on $G'_1$ and $\overline{G}_2$ yields
\begin{eqnarray*}
\ara I(G)&\leq&\ara I(G'_1)+\ara I(\overline{G}_2)\\
&\leq&b'_1+n_1+\overline{b}_2+n_2-1\\
&\leq& b_1+1+n_1+b_2-2+n_2-1\\
&=&b+n.\\
\end{eqnarray*}

{\bf Case 2} Suppose that, for some choice of $C_1$, $x\notin C_1$. If $x\notin C_2$ for some choice of $C_2$, then, by Lemma \ref{lem_covers} (ii), $C_1\cup C_2$ is a maximum minimal vertex cover of $G$, so that $b=b_1+b_2$, and the claim easily follows by induction on $G_1$ and $G_2$. So assume that $x\in C_2$ for all choices of $C_2$.

Perform the following construction, starting at $K=G_2$.
\begin{itemize}
\item[1.] Add to $G_2$ all branches $H$ of $G$ at $x$ such that $x$ belongs to every maximum minimal vertex cover of $H$ (then, by Lemma \ref{lem_covers} (i), the same is true for all maximum minimal vertex covers of the resulting graph $K$);
\item[2.] continue adding branches of $G$ at $x$ as long as $x$ belongs to every maximum minimal cover of the initial graph $K$;
\item[3.] add the remaining branches of $G$ at $x$, and call $K$ the resulting graph.
\end{itemize}

Note that, by assumption, no branch of $G$ at $x$ is a single edge. Each time a branch is added in step 1, induction applies to $K$ and $H$ as to $G_1$ and $G_2$ in Case 1.1. Note that each branch added after step 1 has a maximum minimal vertex cover to which $x$ does not belong. Each time step 2 is performed, induction applies as in Case 1.2, with $K$ and $H$ playing the roles of $G_1$ and $G_2$ respectively.   Each time step 3 is performed on the graph $K$, induction applies to $K$ and $H$  as at the beginning of Case 2. Finally observe that, in view of Lemma \ref{lem_branch}, the algorithm always ends with $K=G$.
\end{proof}

\section{Some special cases}

The examples presented in the final section of \cite{BM2} show that the bound given in Theorem \ref{thm_main} is sharp, in general. Equality holds, in particular, for acyclic graphs, which was shown for the first time in \cite{KT13}. On the other hand, it is also true that in many cases our bound can be improved. We give some classes of cactus graphs of this kind.

The next result is a consequence of our main theorem.

\begin{cor}\label{cor_3}
Let $G$ be a cactus graph and let $n$ be the number of its cycles. Suppose that $k$ cycles are of length divisible by $3$, and each of them has all but two consecutive vertices of degree $2$. Then
\[
\ara I(G) \leq \bight I(G) +n-k.
\]
\end{cor}

\begin{proof}
We proceed by induction on $k$. If $k=0$, then the claim follows directly from Theorem \ref{thm_main}. So assume that $k\geq1$ and the claim is true for all cactus graphs fulfilling the assumption with $k$ replaced by $k-1$. Let $\Gamma$ be a cycle of $G$ of length $s$ divisible by 3. Call $x_1,\dots, x_s$ its consecutive vertices and suppose that the first $s-2$ have degree 2. Then let $G'$ be the graph obtained from $G$ by adding a new vertex $y$ and replacing the edge $x_sx_1$ with $x_sy$. This operation opens the cycle $\Gamma$, but leaves the other cycles of $G$ unchanged. Then $G'$ has $n-1$ cycles, $k-1$ of which fulfil the assumption. Thus induction applies to $G'$, whence
$$\ara I(G')\leq \bight I(G')+n-1-(k-1)=\bight I(G')+n-k.$$
Note that $G$ can be recovered from $G'$ by identifying $x_1$ and $y$. This implies that $\ara I(G)\leq \ara I(G')$. Thus it suffices to prove that $\bight I(G')\leq\bight I(G)$. Let $C'$ be a maximum minimal vertex cover of $G'$. We show that $G$ has a minimal vertex cover $C$ such that $\vert C'\vert \leq\vert C\vert$. Consider the subgraph induced by $G'$ (and $G$) on the set $W=\{x_1, \dots, x_{s-2}\}$, and the cover $D$ induced on it by $C'$; by assumption $D$ is a vertex cover of the path on the vertices $x_1,\dots, x_{s-2}$, minimal if and only if $s=3$ and $x_1\notin C'$, or  $s\geq 6$ and $\{x_{s-3}, x_{s-2}\}\not\subset C'$. If $s\geq 6$, a maximum minimal vertex cover of this path is the set $D_1$ formed by $x_2, x_4$ and the following vertices of $W$ (if any such exist) having index congruent to 1 or 2 modulo 3; we can also take $D_2$, the set formed by $x_1, x_3$ and the following vertices of $W$ (if any such exist) having index congruent to 0 or 2 modulo 3. In both cases we pick 2 vertices and 2 out of 3 from the $s-6$ vertices $x_5,\dots, x_{s-2}$, hence the total number of vertices is $\frac{2(s-2)-2}3$, which is the big height of the edge ideal of the path on $s-2$ vertices if $s$ is divisible by 3 (see, e.g., Corollary 3.0.3 in \cite{M13c}). If $s=3$, we set $D_1=D_2=\emptyset$.  If $D$ is minimal, $\vert D\vert \leq \vert D_1\vert=\vert D_2\vert$. Now set $E=C'\cap\{x_{s-1}, x_s, y\}$. Since $y$ is a terminal vertex of $G'$, exactly one of $y$ and $x_s$ belongs to $E$. Hence $E=\{x_s\}$ or $E=\{x_{s-1}, y\}$ or $E=\{x_{s-1}, x_s\}$. In the first case (which is the only possible if $D$ is not minimal), the minimality of $C'$ is not violated when $x_1$ and $y$ are identified, so that we can take $C=C'$. In the second case, we can take $C=D_1\cup E$ (after replacing $y$ with $x_1$). In the third case we can take $C=D_2\cup E$ if $x_s$ is not a redundant neighbour of $y$ in $C'$. Otherwise, in $C'$ we can replace $x_s$ with $y$, which takes us back to the second case. We thus always have $\vert C'\vert\leq \vert C\vert$, as desired.
\end{proof}

The next propositions are applications of results contained in other papers.
\smallskip

In \cite{M13b} the second author considers the following construction. Given a graph $G$ on $n$ vertices $x_1, \dots, x_n$, one introduces  $n$ new vertices $y_1,\dots, y_n$ and sets $G'=G\cup\{x_1y_1, \dots, x_ny_n\}$; $G'$ is called the {\it whisker graph} on $G$, and is a special type of fully whiskered graph. The edge $x_iy_i$ is called the {\it whisker} at $x_i$. In the same paper (Proposition 2 and Remark 3) it is shown that $\ara I(G')=\hgt I(G')=n$, so that $I(G')$ is a set-theoretic complete intersection. We now generalize this result.

We consider $n$ cycles/edges $H_1,\dots, H_n$ on pairwise disjoint vertex sets, such that, for all $i=1,\dots, n$, $V(G)\cap V(H_i)=\{x_i\}$.  We will say that $H_i$ is {\it attached} to the vertex $x_i$. We then set $G^{\circ}=G\cup\bigcup_{i=1}^nH_i$.  This is, indeed, a natural generalization of the whisker graph $G'$ on $G$.

Please note that the graphs referred to by the next proposition are not always cactus graphs.

\begin{prop}\label{prop_bubble}
Let $G$ be a graph and $G^{\circ}$ be the graph obtained by attaching a cycle or a whisker to each of its vertices. Suppose that $m$ is the number of these cycles whose length is congruent to $1$ modulo $3$. Then
\[
\ara I(G^{\circ}) \leq \bight I(G^{\circ})+m,
\]
In particular, if $m=0$, then $\ara I(G^{\circ})=\bight I(G^{\circ})$. Moreover,   if every attached cycle has length $3$ or $5$, then $I(G^{\circ})$ is a set-theoretic complete intersection.
\end{prop}

\begin{proof}
We will refer to the notation introduced above. For all $i=1,\dots, n$, let $y_i$ be a neighbour of $x_i$ in $H_i$. Then $G'=G\cup\{x_iy_i\vert 1\leq i\leq n\}$ is the whisker graph on $G$. We know that $\ara I(G')=\hgt I(G')=n$. Let $S_0$ be a set of $n$ polynomials of $R$  such that the radical of the ideal $J_0=(S_0)$ is $I(G')$. Then, in particular, $x_iy_i\in \sqrt{J_0}$, for all $i=1,\dots, n$.

For all $i=1,\dots, n$ let $a_i=\ara I(H_i)$ and $b_i=\bight I(H_i)$. We know from \cite{M13c}, Corollary 3.0.5, that $a_i=b_i+1$ if $H_i$ is a cycle of length congruent to 1 modulo 3, otherwise $a_i=b_i$ ($=1$ if $H_i$ is a whisker). In both cases, according to \cite{BKMY12}, Section 2, we can complete the monomial $x_iy_i$ to a set of $a_i$ polynomials that generate an ideal $J_i$ whose radical is $I(H_i)$; let $S_i$ be the set of additional $a_i-1$ polynomials.  Let $J=\sum_{i=0}^n J_i$. Then $J$ has the same radical as the ideal generated by $\bigcup_{i=0}^n S_i$, which admits a set of $n+\sum_{i=1}^n (a_i-1) =\sum_{i=1}^n a_i$ generators. But the radical of $J$ is $I(G^{\circ})=I(G')+\sum_{i=1}^n I(H_i)$. Thus $\ara I(G^{\circ})\leq \sum_{i=1}^n a_i$.

On the other hand, if $C'$ is a maximum minimal vertex cover of $G'$, then $\vert C'\vert=n$ and, for all $i=1,\dots, n$, exactly one of $x_i$ and $y_i$ belongs to $C'$. This vertex can be completed to a maximum minimal vertex cover $C_i$ of $H_i$ (by adding $b_i-1$ vertices). Then $C^{\circ}=C'\cup\bigcup_{i=1}^n C_i$ is a minimal vertex cover of $G^{\circ}$, so that $\bight I(G^{\circ})\geq\sum_{i=1}^nb_i$. Finally, assume that $H_1,\dots, H_m$ are the cycles whose lengths are congruent to 1 modulo 3. We then have

$$\ara I(G^{\circ})\leq \sum_{i=1}^n a_i=\sum_{i=1}^m(b_i+1)+\sum_{i=m+1}^nb_i\leq\bight I(G^{\circ})+m,$$
\noindent
which shows the first part of the claim.  Now, if each $H_i$ is a whisker or a cycle of length 3 or 5, then $m=0$ and, according to \cite{KMY10}, Corollary 2.6, $\bight I(G^{\circ})=\hgt I(G^{\circ})$, so that $\ara I(G^{\circ})=\hgt I(G^{\circ})$.
\end{proof}

Recall that a graph $G$ is called {\it Cohen-Macaulay} if so is its edge ideal on every field. In this case, the ideal $I(G)$ is pure. Moreover, $G$ is always Cohen-Macaulay if $I(G)$ is a set-theoretic complete intersection. It is natural to ask if the converse is true.  In the sequel, we will give several results which provide supporting evidence for a positive answer, especially in the case of cactus graphs.

Recall that a vertex of a graph $G$ is called \textit{simplicial} if its neighbours are pairwise adjacent in $G$. A maximal complete subgraph of $G$ is called a \textit{simplex} if it contains at least one simplicial vertex of $G$.
\begin{cor}\label{cor_chordal}
Let $G$ be a chordal graph or a graph that does not have, among its subgraphs, any cycle of length $4$ or $5$. Then the following conditions are equivalent:
\begin{itemize}
  \item[(a)] $I(G)$ is pure,
  \item[(b)] $G$ is Cohen-Macaulay,
  \item[(c)] every vertex of $G$ belongs to exactly one simplex of $G$,
  \item[(d)] $I(G)$ is a set-theoretic complete intersection.
\end{itemize}
\end{cor}

\begin{proof}
The implications (d) $\Rightarrow$ (b) $\Rightarrow$ (a) are well-known.  If (c) is true, then, with the terminology used in \cite{M13b}, $G$ is a so-called {\it fully clique-whiskered graph}, and (d) follows by Proposition 2 in \cite{M13b}. Finally, implication (a) $\Rightarrow$ (c) follows from the Theorem in \cite{HHZ05} if $G$ is chordal and from Corollary 2.8 in \cite{HMT15} if $G$ has no cycles of length 4 and 5.
\end{proof}

\section{The case of unicyclic graphs}
In this section we consider graphs containing exactly one cycle, which are called \textit{unicyclic}. We prove that the edge ideal of a Cohen-Macaulay unicyclic graph is a set-theoretic complete intersection.
\begin{thm}\label{thm_CM} A connected unicyclic graph $G$ with the unique cycle $\Gamma$ is Cohen-Macaulay if and only if it satisfies one of the following conditions:
\begin{itemize}
\item[1)] $G=\Gamma$ and has length $3$ or $5$;
\item[2)] $G$ is a whisker graph;
\item[3)] $\Gamma$ has length $3$, has some vertex of degree $2$, and the subgraph induced on $V(G) \setminus V(\Gamma)$ is a disjoint union of whisker trees and single edges;
\item[4)] $\Gamma$ has length $5$, has no adjacent vertices of degree greater than $2$,  and the subgraph induced on $V(G) \setminus V(\Gamma)$ is a disjoint union of whisker trees and single edges;
\item[5)] $\Gamma$ has length $4$, has two adjacent vertices of degree $2$, and $G=\Gamma\cup H_1\cup H_2$, where $H_1$ and $H_2$ are non-empty  graphs attached to the remaining two vertices, such that $H_1\cup\{x_1x_2\}\cup H_2$ is a whisker tree.
\end{itemize}
\end{thm}

\begin{proof}
Let $G$ be a unicyclic Cohen-Macaulay graph. Then $G$ is unmixed and $G \neq C_4, C_7$. Hence, by \cite[Theorem 3]{TV90}, $G$ is one of the graphs described in 1) - 5).

Conversely, if $G$ is one of the graphs in 1) - 5), then $G$ belongs to the class $\mathcal{SQC}$ defined in \cite{HMT15}, p. 4. Hence, by Theorem 2.3 in \cite{HMT15}, it follows that $G$ is Cohen-Macaulay.
\end{proof}

We recall that the set-theoretic complete intersection property is true in case 1) by \cite{BKMY12}, Corollary 1. In case 2) it follows from \cite{M13b}, Proposition 3 and Remark 3.
In case 3) it is a consequence of Corollary \ref{cor_3}. The proof for the remaining two cases will be performed in several steps.

We need to recall two technical facts about the sets of polynomials that generate a monomial ideal up to radical:
\begin{list}{}{}
\item{(i)}  if $\mu$, $\nu$ and $\rho$ are monomials such that $\rho$ divides $\mu\nu$, then $\sqrt{(\rho, \mu+\nu)}=\sqrt{(\rho, \mu , \nu)}$;
\item{(ii)} if $\mu$ is any non-terminal edge of a whisker graph on $n$ vertices, then it can be completed to a set of $n$ polynomials generating its edge ideal up to radical (note that $n$ is the number of non-terminal vertices of the whisker graph; these form a maximum minimal vertex cover of the whisker graph).
\end{list}

Fact (i) is the main idea behind a famous result by Schmitt and Vogel (see the Lemma in \cite{SV79}). Fact (ii) is part of the proof of Proposition 3 in \cite{M13b}, which is based on the construction presented in the proof of Theorem 1 in \cite{K09}.
\smallskip

Let $\Gamma_5$ be the cycle on the consecutive vertices $x_1, x_2, x_3, x_4, x_5$. The graph $G'$ considered in the second of the following two lemmas is the one presented in part 4) of Theorem \ref{thm_CM}. The graph $G$ considered in the first lemma is a special case of it. In particular, $I(G)$ and $I(G')$ are pure.

\begin{lem}\label{lem_5cycle} For some nonnegative integers $r,s$, let $H_1, \dots, H_r$ and $K_1,\dots, K_s$ be paths of length $2$ attached to the vertices $x_1$ and $x_3$ of $\Gamma_5$, respectively, and let $G$ be the union of $\Gamma_5$, $\bigcup_{i=1}^r H_i$ and $\bigcup_{i=1}^s K_i$. Then $\ara I(G)=\hgt I(G)=r+s+3$. In particular, $I(G)$ is a set-theoretic complete intersection.
\end{lem}

\begin{proof}  For all $i=1,\dots, r$ let $x_1, a_i, b_i$ be the consecutive vertices of $H_i$, and for all $i=1,\dots s$, let $x_3, c_i, d_i$ be the consecutive vertices of $K_i$. A maximum minimal vertex cover of $G$ is
$C=\{x_1, x_3, x_4, b_1, \dots, b_r, d_1, \dots, d_s\}$. We show that $I(G)$ can be generated up to radical by $r+s+3$ elements. We first present some special cases.

From Proposition 1 (c) in \cite{BKMY12} we know that
$$I(\Gamma_5)=\sqrt{(x_1x_2, x_2x_3+x_4x_5, x_1x_5+x_3x_4)},$$
 which proves the claim for $r=s=0$.  Now suppose that $r\geq 1$, and, for all $i=1,\dots, r-1$, set $h_i=x_1a_{i+1}+a_ib_i$.  Note that, by (i), for $i=1,\dots, r$,  the monomials $x_1a_i$, and for $i=1,\dots, r-1$, the monomials $a_ib_i$  belong to $\sqrt{(x_1a_1, h_1, \dots, h_{r-1})}$. Hence, from (i) it follows that
$$I(\Gamma_5\cup \textstyle{\bigcup_{i=1}^r} H_i)=\sqrt{(x_1a_1, h_1, \dots, h_{r-1}, x_1x_2+a_rb_r, x_2x_3+x_4x_5, x_1x_5+x_3x_4)},$$
\noindent
which shows the claim for $s=0$. Finally, suppose that $s\geq 1$ and for all $i=1,\dots, s-1$, set $k_i=x_3c_{i+1}+c_id_i$. Then set
$$\ p_1=x_1x_2+a_rb_r+x_4x_5,\ p_2=x_1x_5+x_2x_3+a_rx_4x_5,$$
\noindent
and consider the ideal $J=(x_1a_1, h_1, \dots, h_{r-1}, p_1, p_2, x_3c_1, k_1, \dots, k_{s-1}, c_sd_s+x_3x_4)$.  From (i) one deduces that  for all $i=1,\dots, s$,   the monomials $x_3c_i$, $c_id_i$ and the monomial $x_3x_4$  belong to  $\sqrt{(x_3c_1, k_1, \dots, k_{s-1}, c_sd_s+x_3x_4)}$. Moreover,
$$x_1^2x_2-a_rx_4^2x_5=-b_rx_1a_r+x_1p_1-x_4p_2+x_2x_3x_4\in \sqrt{J},$$
\noindent
which, together with $x_1a_r\in \sqrt{J}$, and, in view of (i), yields $x_1x_2, a_rx_4x_5\in \sqrt{J}$. Thus $a_rb_r+x_4x_5\in\sqrt{J}$, which, similarly, implies that $a_rb_r, x_4x_5\in \sqrt{J}$. It then easily follows from (i) that
$$I(G)=\sqrt{J},$$
\noindent
which completes the proof of the claim.
\end{proof}

Let $G$ be the graph in Lemma \ref{lem_5cycle}.

\begin{lem}\label{lem_5cycleextended}  Let $G'$ be the graph obtained from $G$ by attaching some graphs $L_i$ to $x_1$ and/or some graphs $M_i$ to $x_3$ where, for all $i$, the subgraph $\overline{L}_i$ induced by $L_i$ on $V(L_i)\setminus\{x_1\}$ and the subgraph $\overline{M}_i$ induced by $M_i$ on $V(M_i)\setminus\{x_3\}$ are whisker trees.  Then $I(G')$ is a set-theoretic complete intersection.
\end{lem}

\begin{proof} Fix an index $i$ and set $L=L_i$. Since $\overline{L}=\overline{L}_i$ is connected acyclic and is not reduced to a single edge, $x_1$ has in $\overline{L}$ only one neighbour $e=e_i$, which has some neighbour $f\ne x_1$ that is not a terminal vertex.  Let $H$ be the path formed by $x_1, e, f$. By Lemma \ref{lem_5cycle}, $I(G\cup H)$ is a set-theoretic complete intersection. We show that this property also holds for $I(G\cup L)$, i.e., it is preserved when the edge $ef$ is completed to $\overline{L}$. As was observed in the proof of Lemma \ref{lem_5cycle}, a maximum minimal vertex cover of $G\cup H$ is $C=\{x_1, x_3, x_4, b_1, \dots, b_r, d_1, \dots, d_s, f\}$, which induces on $H$ the cover $\{x_1, f\}$.    On the other hand, a maximum minimal vertex cover of $\overline{L}$ is the set $\overline{D}$ of all its non-terminal vertices (among which is $f$). Then $C\cup\overline{D}$ is a minimal vertex cover of $G\cup L$.  Since $C\cap \overline{D}=\{f\}$, $\vert C\cup \overline{D}\vert=\vert C\vert+\vert\overline{D}\vert-1$. We know that $I(\overline{L})$ is a set-theoretic complete intersection. First assume that $e$ is not a terminal vertex of $\overline{L}$. Then, in view of (ii), there exist $\ell_1, \dots, \ell_{\vert\overline{D}\vert-1}\in R$ such that $I(\overline{L})=\sqrt{(ef, \ell_1, \dots, \ell_{\vert\overline{D}\vert-1})}.$ On the other hand there are $g_1, \dots, g_{\vert C\vert}\in R$ such that $I(G\cup H)=\sqrt{(g_1, \dots, g_{\vert C\vert})}$. Since $ef\in I(G\cup H)$, it follows that
$$I(G\cup L)=I(G\cup H\cup\overline{L})=\sqrt{(g_1, \dots, g_{\vert C\vert}, \ell_1, \dots, \ell_{\vert\overline{D}\vert-1})},$$
which proves the claim in this case.
Repeat the above construction for all graphs $L=L_i$ such that $e_i$ is not a terminal vertex of $\overline{L}$ (say, for all $i=1,\dots, u$) and for all graphs $M_i$ in which the only neighbour $g_i$ of $x_3$ is not a terminal vertex of $\overline{M}_i$ (say, for all $i=1,\dots, v$). This will produce a graph $G_1$ with a maximum minimal vertex cover containing
$$\{x_1, x_3, x_4, b_1, \dots, b_r, d_1, \dots, d_s, e_1, \dots, e_u, g_1, \dots, g_v\}.$$
\noindent
We obtain a new maximum minimal vertex cover $C_1$ if we replace
$$\{x_1, x_3, b_1, \dots, b_r, d_1, \dots, d_s\}\ \ \ \mbox{  by }\ \ \  \{x_2, x_5, a_1, \dots, a_r, c_1, \dots, c_s\}.$$
\noindent
Now suppose that $e$ is a terminal vertex of $\overline{L}$ (i.e., $ef$ is the whisker attached to $f$). In $\overline{D}$ replace each non-terminal neighbour of $f$ (i.e., each neighbour of $f$ other than $e$) with the endpoint of the whisker attached to it. This yields a new maximum minimal vertex cover of $\overline{L}$: note that, since $\overline{L}$ is acyclic, no two neighbours of $f$ are adjacent, so that minimality is preserved. Then add the vertex $e$: in this way we obtain a minimal vertex cover $D$ of $L$, which is maximum by Lemma 3.2 in \cite{BM2}.   Then the disjoint union $C_1\cup D$ is a minimal vertex cover of $G_1\cup L$. Now, since $L$ is a tree,  $I(L)$ is generated up to radical by $\vert D\vert$ elements. Thus $I(G\cup L)$ is generated up to radical by $\vert C_1\vert +\vert D\vert$ elements, which completes the proof of the set-theoretic complete intersection property for $I(G_1\cup L)$.\\
\noindent An iteration of the previous construction yields the claim for $G'$.
\end{proof}

Let $\Gamma_4$ be the cycle on the consecutive vertices $x_1, x_2, x_3, x_4$. The graph $G$ in the following lemma is the one described in case 5) of Theorem \ref{thm_CM}.

\begin{lem}\label{lem_4cycle}  For $i=1,2$ let $H_i$ be a non-empty tree  attached to the vertex $x_i$ of $\Gamma_4$ and suppose that $H_1\cup\{x_1x_2\}\cup H_2$ is a whisker tree. Let $G=\Gamma_4\cup H_1\cup H_2$. Then $I(G)$ is a set-theoretic complete intersection.
\end{lem}

\begin{proof}
First of all note that, in view of (i), the set $S_0=\{x_1x_2, x_1x_4+x_2x_3,x_3x_4+x_1y_1+x_2y_2\}$ generates  the edge ideal of $\Gamma_4\cup\{x_1y_1, x_2y_2\}$ up to radical. Set $H=H_1\cup\{x_1x_2\}\cup H_2$.

Fix an index $i$. Then $H_i$ is either a single edge $x_iy_i$ or a whisker tree. In the latter case $x_i$ has in $H_i$ some non-terminal neighbour $y_i$ (since in a whisker graph no vertex is adjacent to more than one terminal vertex). If $H_i$ is a single edge, then $C_i=\{x_i\}$ is a minimal vertex cover of $H_i$. Otherwise a (maximum) minimal vertex cover $C_i$ of $H_i$ consists of all non-terminal vertices of $H_i$.  Note that in this case $x_i$ is not a terminal vertex in $H_i$: otherwise it would be the whisker at $y_i$ in $H_i$, but then $y_i$ would not have any whisker in $H$, where $x_i$ is not a terminal vertex. Thus $x_iy_i$ is not a terminal vertex of $H_i$.

In any case, the set $D=C_1 \cup C_2 \cup \{x_3\}$ is a minimal vertex cover of $G$.   Now, in view of (ii), there is a set $S_i$ of $\vert C_i\vert-1$ polynomials such that $I(H_i)=\sqrt{(x_iy_i, S_i)}$ (this is trivially true if $H_i$ is a single edge, with $S_i=\emptyset$).  In view of (i), we thus have that $I(G)=\sqrt{(S_0 \cup S_1 \cup S_2)}$, where $\vert S_0\vert+\vert S_1\vert+\vert S_2\vert= \vert C_1\vert+\vert C_2\vert+1$, which proves the claim.
\end{proof}

We have thus completed the proof of the announced result.

\begin{thm}\label{thm_1cycle}
Let $G$ be a Cohen-Macaulay graph with exactly one cycle. Then $I(G)$ is a set-theoretic complete intersection.
\end{thm}

\section{Final Remarks}
The assumptions of Corollary \ref{cor_chordal} could suggest that the presence of minimal cycles of length 4 or 5 may in general provide an obstruction to the set-theoretic complete intersection property in the Cohen-Macaulay case, in the same way in which it sometimes prevents to improve the bound given in Theorem \ref{thm_main} (which, in fact, is what occurs in the example presented at the end of \cite{BM2}). Also the next result, while strongly enforcing the idea that Cohen-Macaulayness is ``close'' to set-theoretic complete intersection, clearly hints at this kind of obstruction.

\begin{cor}
Let $G$ be a connected graph, which is not a single edge nor a cycle of length 7, and has no minimal cycles of length less than 6. Then the following conditions are equivalent:
\begin{itemize}
  \item[(a)] $I(G)$ is pure,
  \item[(b)] $G$ is a whisker graph,
  \item[(c)] $G$ is Cohen-Macaulay,
  \item[(d)] $I(G)$ is a set-theoretic complete intersection.
\end{itemize}
\end{cor}

\begin{proof}
The equivalence (a) $\Leftrightarrow$ (b) is Corollary 5 in \cite{FHN93}.  Moreover,  (b) $\Rightarrow$ (d) follows from \cite{M13b} Proposition 3. The implications (d) $\Rightarrow$ (c) and (c) $\Rightarrow$ (a) are well known.
\end{proof}

\begin{Acknowledgements}
The authors are indebted to Russ Woodroofe for raising the question that inspired this paper.
\end{Acknowledgements}

\end{document}